\documentclass{article}

\usepackage{amsmath, amsthm}
\usepackage{amssymb}
\usepackage{color}
\usepackage{graphicx}
\usepackage{pstricks}
\usepackage{pst-plot}
\usepackage{tikz}
\usetikzlibrary{intersections}

\newcommand{\CC}{\mathbb{C}}

\newcommand{\RR}{\mathbb{R}}

\newcommand{\Paivarinta}{$\mathrm{P\ddot{a}iv\ddot{a}rinta}$}

\newtheorem{thm}{Theorem}[section]
\newtheorem{prop}[thm]{Proposition}
\newtheorem{lemma}[thm]{Lemma}

\newtheorem{coro}[thm]{Corollary}

\title{Inverse Problem of Electro-seismic Conversion}
\author{Jie Chen and Yang Yang}
\date{}

\begin{document}

\maketitle

\begin{abstract}
When a porous rock is saturated with an electrolyte, electrical fields are coupled with seismic waves via the electro-seismic conversion. Pride \cite{Pride1994} derived the governing models, in which Maxwell equations are coupled with Biot's equations through the electro-kinetic mobility parameter. The inverse problem of the linearized electro-seismic conversion consists in two step, namely the inversion of Biot's equations and the inversion of Maxwell equations. We analyze the reconstruction of conductivity and electro-kinetic mobility parameter in Maxwell equations with internal measurements, while the internal measurements are provided by the results of the inversion of Biot's equations. We show that knowledge of two internal data based on well-chosen boundary conditions uniquely determine these two parameters. Moreover, a Lipschitz type stability is proved based on the same sets of well-chosen boundary conditions.
\end{abstract}

\section{Introduction}

When a porous rock is saturated with an electrolyte, an electric double layer is formed at the interface of the solid and the fluid. One side of the interface is negatively charged and the other side is positively charged. Such electric double layer(EDL) system is also called Debye layer. Due to the EDL system, electromagnetic(EM) fields and mechanical waves are coupled through the phenomenon of electro-kinetics. Precisely, electrical fields or EM waves acting on the EDL will move the charges, creating relative movement of fluid and solid. This is called electro-seismic conversion. Conversely, mechanical waves moving fluid and solid will generate EM fields. This is called seismo-electric conversion. Thompson and Gist \cite{Thom1993} have made field measurement clearly demonstrating seismo-electric conversion in saturated sediments. Zhu et al. \cite{Zhu1999, Zhu2003, Zhu2005} made laboratory experiments and observed
the seismo-electric conversion in model wells, and their experimental results confirm that seismo-electric logging could be a new bore-hole logging technique. 

The investigation of wave propagation in fluid-saturated porous media was early developed by Biot \cite{Biot1956, Biot1956_2}. The governing equations of the electro-seismic converstion was derived by Pride \cite{Pride1994} as following.
\begin{eqnarray}
	&\nabla\times E = i\omega\mu H,&
			\label{Maxwell1}\\
	&\nabla\times H = (\sigma - i\epsilon\omega)E 
			+ L(-\nabla p +\omega^2\rho_f u) + J_s,&
			\label{Maxwell2}\\
	&-\omega^2(\rho u + \rho_f w) = \nabla\cdot \tau,&
			\label{Biot1}\\
	&-i\omega w = LE + \frac{\kappa}{\eta}(-\nabla p + \omega^2\rho_f u),&
			\label{Biot2}\\
	&\tau = (\lambda\nabla\cdot u + c\nabla\cdot w)I 
			+ G(\nabla u + \nabla u^T),&
			\label{Biot3}\\
	&-p = c\nabla\cdot u + M\nabla\cdot w,&
			\label{Biot4}
\end{eqnarray}
where the first two are Maxwell's equations, the remaining are Biot's equations. The notation is as follows:
\begin{itemize}
	\item[$E$]        electric field,
	\item[$H$]        magnetizing field or magnetic field intensity,
	\item[$\omega$]   seismic wave frequency,
	\item[$\sigma$]   conductivity,
	\item[$\epsilon$] dielectric constant or relative permittivity,
	\item[$\mu$]      magnetic permeability,
	\item[$J_s$]      source current,
	\item[$p$]        pore pressure,
	\item[$\rho_f$]   density of pore fluid,
	\item[$L$]        electro-kinetic mobility parameter,
	\item[$\kappa$]   fluid flow permeability,
	\item[$u$]        solid displacement,
	\item[$w$]        fluid displacement,
	\item[$\tau$]     bulk stress tensor,
	\item[$\eta$]     viscosity of pore fluid,
	\item[$\lambda, G$] Lam$\acute{\mathrm{e}}$ parameters of elasticity,
	\item[$C, M$]     Biot moduli parameters.
\end{itemize}
Pride and Haartsen \cite{Pride1996} also analyzed the basic properties of seismo-electric waves.

Notice that the coupling is non-linear, namely electro-seismic and seismo-electric conversions happen simultaneously. Under the assumptions that the coupling is so weak that multiple coupling is neglectable, we can linearize the forward system in two steps. Particularly, we focus on the electro-seismic conversion and ignore the seismo-electric conversion. The first step in the forward system is modeled by Maxwell equations without the effect of the seismic waves, i.e., $L=0$ in \eqref{Maxwell2}. While the electro-seismic conversion happens, the seismic waves are generated and modeled by Biot's equations with potential $LE$ in \eqref{Biot2}. 

In the present paper, we mainly focus on the inverse problem of the linearized electro-seismic conversion, which is a hybrid problem and consists of two steps. The first step of the inverse problem is to invert Biot's equations, i.e., to recover the potential $LE$ in \eqref{Biot2} from any measurements observed on the domain boundary. Williams \cite{Williams2001} presented an approximation to Biot's equations, which could be a useful tool to study the inverse problem.

Assuming the first step is implemented successfully, the second step of the inverse problem is to invert Maxwell's equations, which consists of reconstructing the conductivity $\sigma$ and the electro-kinetic mobility parameter or the coupling coefficient $L$ from boundary measurements of the electrical fields and the internal data $LE$ obtained in the first step.

The problem of interest in this paper is the second step of the inverse problem. We study the reconstruction of the conductivity $\sigma$ and the coupling coefficient $L$ and prove uniqueness and stability results of the reconstructions. Particularly, we show that $\sigma, L$ are uniquely determined by $2$ well-chosen electrical fields at the domain boundary. The explicit reconstruction procedure is presented. The stability of the reconstruction is established from either $2$ measurements under geometrical conditions or from $6$ well-chosen boundary conditions.

Mathematically, our proof relies on explicit solutions to Maxwell's equations, namely Complex Geometrical Optics (CGO) solutions, constructed by Colton and \Paivarinta \cite{Colton1992}. In our reconstruction procedure, the coupling coefficient $L$ satisfies a transport equation with vector field $\beta$. With CGO solutions, we can prove the integral curves of the vector field $\beta$ are close to straight lines and exit the domain in finite time. Therefore, $L$ can be uniquely and explicitly solved by the characteristic method. Stability follows the analysis of the method of characteristic.

The rest of the paper is structured as follows. Section \ref{se:Main} presents our main results. The CGO solutions are introduced in section \ref{se:CGO}. The inverse Maxwell's equations and an explicit reconstruction algorithm are addressed in the rest of section \ref{se:Maxwell}, while section \ref{se:unique} focusing on the proof of the uniqueness result and section \ref{se:stab} and \ref{se:stab2} focusing on the stability proof.

\section{Main results}\label{se:Main}

Let $\Omega$ be an open, bounded and connected domain in $\RR^3$ with $C^2$ boundary $\partial\Omega$. In the second step of the electro-seismic conversion, the propagation of the electrical fields is modeled by Maxwell's equations in $\Omega$, 
\begin{equation}\left\{
\begin{array}{rcl}\label{Maxwell}
	\nabla\times E &=& i\omega\mu H,\\
	\nabla\times H &=& (\sigma - i\epsilon\omega)E + J_s.
\end{array}
\right.\end{equation}
The measurements available for the inverse problem include the internal data from the first step
\begin{equation}
	D:=LE,\;\mathrm{in}\;\Omega
\end{equation}
and the boundary illumination, i.e., the tangential boundary measurement of the electrical field
\begin{equation}
	G:=tE,\;  \mathrm{on}\;\partial\Omega.
\end{equation}
Define the operator
\begin{equation} \label{Max_Meas_map}
	\Lambda_M (L,\sigma) := (J_s, D, G).
\end{equation}
The problem now is to invert the operator $\Lambda_M$, or namely, to reconstruct $(L,\sigma)$ from some measurements $(J_{s,j}, D_j, G_j)$ indexed by $j$, assuming $\mu$ and $\epsilon$ are given.

The main purpose of this paper is to prove the uniqueness and stability of the coefficient reconstructions. We define the set of coefficients $(L, \sigma)\in\mathcal{M}$ as
\begin{align}
\label{eq:para_space}
\mathcal{M} = \{ (L,\sigma)\in & \;C^d\times C^{d-2}: d\geq 2 \\ &\mathrm{and}\; 0 \mathrm{\;is\;not\;an\;eigenvalue\;of\;} \nabla\times\nabla\times\cdot-k^2n\},\notag
\end{align}
where the wave number $k>0$ and the refractive index $n$ are given by
\begin{equation}
\label{Maxwellkn}
	k = \omega\sqrt{\epsilon_0\mu_0},\quad
	n = \frac{1}{\epsilon_0}\left(\epsilon + i\frac{\sigma}{\omega}\right).
\end{equation}

The main results are as follows, where the measurements $G$ and $D$ are complex-valued.

\begin{thm}\label{thm:unique}
Let $\Omega$ be an open, bounded subset of $\RR^3$ with boundary $\partial\Omega$ of class $C^{d}$. Let $(L,\sigma)$ and $(\tilde{L}, \tilde{\sigma})$ be two elements in $\mathcal{M}$. Let $D:=(D_1,D_2)$ and $\tilde{D}:=(\tilde{D}_1,\tilde{D}_2)$, be two sets of internal data on $\Omega$ for the coefficients $(L,\sigma)$, $(\tilde{L},\tilde{\sigma})$, respectively and with boundary illuminations $G:=(G_1,G_2)$. 

Then there is a subset of $G \in (C^{d+3}(\partial\Omega))^2$, such that if $D_j=\tilde{D}_j$, $j=1,2$, we have $(L,\sigma) = (\tilde{L},\tilde{\sigma})$.
\end{thm}

To consider the stability of the reconstruction, we need to restrict to a subset of $\Omega$. Let $\zeta_0$ be a constant unit vector. Let $x_0\in\partial\Omega$ be the tangent point of $\partial\Omega$ with respect to $\zeta_0$, i.e., the tangent line of $\partial\Omega$ at $x_0$ is parallel to $\zeta_0$. Define $\Omega_1$ to be the subset of $\Omega$ by removing a neighborhood of each tangent point $x_0\in\partial\Omega$.


\begin{thm}\label{thm:stab}
let $d\geq 3 $. Let $\Omega$ be convex with $C^d$ boundary $\partial\Omega$ and $\Omega_1$ is defined as above. Assume that $(L, \sigma)$ and $(\tilde{L},\tilde{\sigma})$ are two elements in $\mathcal{M}$. Let $D = (D_j)$ and $\tilde{D} = (\tilde{D}_j)$, $j=1,2$, be the internal data for coefficients $(L, \sigma)$ and $(\tilde{L},\tilde{\sigma})$, respectively, with boundary conditions $G=(G_j)$, $j=1,2$. 

Then there is a set of illuminations $G\in (C^{d+3}(\partial\Omega))^2$ such that restricting to $\Omega_1$, we have
\begin{equation}
\label{eq:stable}
	\|L-\tilde{L}\|_{C^{d-1}(\Omega_1)} + \|\sigma-\tilde{\sigma}\|_{C^{d-3}(\Omega_1)} \leq C\|D-\tilde{D}\|_{(C^{d}(\Omega_1))^2}.
\end{equation}
\end{thm}

The geometric conditions can be removed when more measurements are available. In particular, when 6 complex measurements are provided, we have the following stability result.
\begin{thm}\label{thm:stab2n}
let $d\geq 3$. Let $\Omega$ be convex and $\Omega_1$ is defined as above. Assume that $(L, \sigma)$ and $(\tilde{L},\tilde{\sigma})$ are two elements in $\mathcal{M}$. Let $D = (D_1^j, D_2^j)$ and $\tilde{D} = (\tilde{D}_1^j,\tilde{D}_2^j)$, $j=1,2,3$, be the internal data for coefficients $(L, \sigma)$ and $(\tilde{L},\tilde{\sigma})$, respectively, with boundary conditions $G=(G_1^j, G_2^j)$, $j=1,2,3$. 

Then there is a set of illuminations $G\in (C^{d+3}(\partial\Omega))^6$, such that
\begin{equation}
\label{eq:stable9}
	\|L-\tilde{L}\|_{C^{d-1}(\Omega)} + \|\sigma-\tilde{\sigma}\|_{C^{d-3}(\Omega)} \leq C\|D-\tilde{D}\|_{(C^{d}(\Omega))^6}.
\end{equation}
\end{thm}

Note that the above measurements are all complex-valued. We will need two real measurements to make up one complex data.

\section{Inversion of Maxwell's equations with internal data}\label{se:Maxwell}
Let $\Omega$ be an open, bounded and connected domain in $\RR^3$ with $C^2$ boundary $\partial\Omega$. In the case when $\mu \equiv \mu_0$ is constant and $J_s = 0$ in $\Omega$, we can rewrite the system in \eqref{Maxwell} as
\begin{equation} \label{Maxwell_E1}
  \nabla\times\nabla\times E - k^2n E =0,
\end{equation}
and
\begin{equation} \label{Maxwell_E2}
	\nabla\cdot nE = 0
\end{equation}
where the wave number $k>0$ and the refractive index $n$ are given by \eqref{Maxwellkn}.

\subsection{Complex Geometrical Optics solutions}\label{se:CGO}
Colton and \Paivarinta \cite{Colton1992} constructed explicit solutions, namely \emph{Complex Geometrical Optics} solutions(CGOs), to the Maxwell's equation \eqref{Maxwell_E1} and \eqref{Maxwell_E2}. CGOs will be the main technique we will use to solve the inverse Maxwell problem. We follow the construction of CGOs in \cite{Colton1992} and extend their properties from $L^2$ space to higher order sobolev spaces. CGOs are of the form
\begin{equation}
	\label{CGOs}
	E(x) = e^{i\zeta\cdot x}(\eta + R_\zeta(x)),
\end{equation}
where $\zeta\in\CC^3\backslash\RR^3$, $\eta\in\CC^3$, are constant vectors satisfying
\begin{equation}
	\label{CGO_para_con}
	\zeta\cdot\zeta = k^2, \;
	\zeta\cdot\eta = 0.
\end{equation}
Substituting \eqref{CGOs} into \eqref{Maxwell_E1} and \eqref{Maxwell_E2} gives
\begin{eqnarray}
\label{Maxwell3} \tilde{\nabla}\times\tilde{\nabla}\times R_\zeta &=& k^2(n-1)\eta + k^2 n R_\zeta,\\
\label{Maxwell4} \tilde{\nabla}\cdot R_\zeta &=& -\alpha\cdot (\eta + R_\zeta)
\end{eqnarray}
where $\tilde{\nabla} := \nabla + i\zeta$ and $\alpha := \nabla n(x)/n(x)$. We further define $\tilde{\Delta} := \Delta + 2i\zeta\cdot \nabla -k^2$. By substituting the formula $\tilde{\nabla}\times\tilde{\nabla} \times R_\zeta = -\tilde{\Delta} R_\zeta +\tilde{\nabla}\tilde{\nabla}\cdot R_\zeta$ into \eqref{Maxwell3} and \eqref{Maxwell4}, we see that $R_\zeta$ is a solution to
\begin{equation}
\label{Maxwell5}
(\Delta + 2i\zeta\cdot\nabla) R_\zeta = -\tilde{\nabla}(\alpha\cdot (\eta + R_\zeta)) + k^2(1-n)(\eta + R_\zeta).
\end{equation}

It was proved in \cite{Colton1992} the existence of $R_\zeta$ to \eqref{Maxwell5} as a $C^2(\RR^3)$ functions. For our analysis, we need to extend the results of CGOs in \cite{Colton1992} to smoother function spaces.

Let the space $L^2_\delta$ for $\delta\in\RR$ be the completion of $C^\infty_0(\RR^3)$ with respect to the norm $\|\cdot\|_{L^2_\delta}$ defined as
\begin{equation}
\|u\|_{L^2_\delta} = \left(\int_{\RR^3}\langle x\rangle^{2\delta} |u|^2 dx\right)^{\frac{1}{2}},\qquad \langle x\rangle = (1 + |x|^2)^{\frac{1}{2}}.
\end{equation}
To get smoother CGOs than that constructed in \cite{Colton1992}, we introduce the space $H^s_\delta$ for $s>0$ as the completion of $C^\infty_0(\RR^3)$ with respect to the norm $\|\cdot\|_{H^s_\delta}$ defined as
\begin{equation}
\|u\|_{H^s_\delta} = \left(\int_{\RR^3}\langle x\rangle^{2\delta} |(I-\Delta)^{\frac{s}{2}}u|^2 dx\right)^{\frac{1}{2}}
\end{equation}
Here $(I-\Delta)^{\frac{s}{2}}u$ is defined as the inverse Fourier transform of $\langle \xi\rangle^s\hat{u}(\xi)$, where $\hat{u}(\xi)$ is the Fourier transform of $u(x)$. 

We recall \cite{Uhlmann1987} for $|\zeta|\geq c>0$ and $v\in L^2_{\delta+1}$ with $-1<\delta<0$, the equation
\begin{equation}
	(\Delta + 2i\zeta\cdot \nabla)u = v
\end{equation}
admits a unique week solution $u\in L^2_\delta$ with
\begin{equation}
	\|u\|_{L^2_\delta}\leq C(\delta, c)\frac{\|v\|_{L^2_{\delta+1}}}{|\zeta|}.
\end{equation}
Since $(\Delta + 2i\zeta\cdot \nabla)$ and $(I-\Delta)^s$ are constant coefficient operators and hence commute, we deduce that when $v\in H^s_{\delta+1}$, for $s>0$, then
\begin{equation}
	\|u\|_{H^s_\delta}\leq C(\delta, c)\frac{\|v\|_{H^s_{\delta+1}}}{|\zeta|}.
\end{equation}
We define the integral operator $G_\zeta:H^s_{\delta+1}(\RR^3) \rightarrow H^s_\delta(\RR^3)$ by
\begin{equation}
	G_\zeta(v) := F^{-1}\left(\frac{\hat{v}}{\xi^2 + 2\zeta\cdot\xi}\right),
\end{equation}
where $F^{-1}$ is the inverse Fourier transform. We see that $G_\zeta$ is bounded and there exists a positive constant $C(\delta)$ such that
\begin{equation}
\label{G_bd}
\|G_\zeta\|\leq\frac{C}{|\zeta|}.
\end{equation}

Before we can prove the existence of a unique solution to \eqref{Maxwell5}, we first prove the following lemma.
\begin{lemma} \label{thm:lemma1}
For any $v\in H^s_{\delta+1}(\RR^3)$ and $|\zeta|$ sufficiently large, the equation
\begin{equation}
\label{Maxwell6}
	(\Delta + 2i\zeta\cdot\nabla + \alpha\cdot\tilde\nabla) u = v
\end{equation}
has a unique solution $u\in H^s_\delta(\RR^3)$ satisfying
\begin{equation}
	\|u + n^{-1/2}G_\zeta(n^{1/2}v)\|_{H^s_\delta} \leq \frac{C}{|\zeta|^2},
\end{equation}
for some positive constant $C$ independent of $\zeta$.
\end{lemma}

The Lemma 3.1 in \cite{Colton1992} proves for the case when $s=0$. We study any $s>0$ here.
\begin{proof}
From the identity
\begin{equation}
\label{lemma_transf}
	n^{-1/2}(\Delta + 2i\zeta\cdot\nabla)(n^{1/2}u) = (\Delta + 2i\zeta\cdot\nabla + \alpha\cdot\tilde\nabla) u + qu,
\end{equation}
where $q:= \Delta n^{1/2}/n^{1/2}$, we can rewrite \eqref{Maxwell6} as
\begin{equation}
	(\Delta + 2i\zeta\cdot\nabla - q)f = g,
\end{equation}
where $f:=n^{1/2}u$ and $g:=n^{1/2}v$. Applying the integral operator $G_\zeta$ gives
\begin{equation}
\label{lemma_f}
	f + G_\zeta(qf) = -G_\zeta(g),
\end{equation}
which admits a unique solution in $H^s_\delta(\RR^3)$ since $I+G_\zeta(q\cdot)$ is invertible for $|\zeta|$ sufficiently large. Eq. \eqref{G_bd} also gives
\begin{equation}
\|f + G_\zeta(g)\|_{H^s_\delta} = \|G_\zeta(q(G_\zeta(qf) +G_\zeta(g)))\|_{H^s_{\delta}} \leq \frac{C}{|\zeta|^2},
\end{equation}
for some positive constant $C$ independent of $\zeta$. This proves the lemma.
\end{proof}

\begin{prop}
\label{thm:CGO}
For $|\zeta|$ sufficiently large, there is a unique solution $R_\zeta\in H^s_\delta(\RR^3)$ to \eqref{Maxwell5}. Thus, the CGO solution $E$ defined by \eqref{CGOs} satisfies \eqref{Maxwell_E1} and \eqref{Maxwell_E2}. Moreover, $R_\zeta$ satisfies 
\begin{equation}
\label{Maxwell9}
	\|R - in^{-1/2}G_\zeta(n^{1/2}\alpha\cdot\eta)\zeta\|_{H^s_\delta} = O\left(\frac{1}{|\zeta|}\right).
\end{equation}
\end{prop}

\begin{proof}
By applying the vector identity
\begin{equation}
	\nabla(A\cdot B) = A\times(\nabla\times B) + B\times(\nabla\times A) + (A\cdot \nabla)B + (B\cdot\nabla)A,
\end{equation}
and rearranging terms, we can rewrite \eqref{Maxwell5} as
\begin{align}
	(\Delta + 2i\zeta\cdot\nabla + \alpha\cdot\tilde{\nabla})R_\zeta = & -\alpha\times(\tilde\nabla\times R_\zeta)  \label{Maxwell7}	- (R_\zeta\cdot\nabla)\alpha \\&- \tilde{\nabla}(\alpha\cdot\eta) + k^2(1-n)(\eta + R_\zeta).\notag
\end{align}
Recall that $\tilde{\nabla} = \nabla + i\zeta$, which is potentially troublesome. We assume that $Q:=\tilde{\nabla}\times R_\zeta$. By \eqref{Maxwell3}, we have that
\begin{equation}
	\tilde{\nabla}\times Q = k^2(n-1)\eta + k^2n R_\zeta
\end{equation}
and hence
\begin{equation}
	\tilde{\nabla}\times\tilde{\nabla}\times Q = k^2\nabla n\times (\eta + R_\zeta) +  k^2(n-1)\zeta\times\eta  + k^2 n Q.
\end{equation}
Since $\tilde{\nabla}\cdot Q = 0$, we now have 
\begin{equation}
\label{prop:Q}
	\Delta Q =  2i\zeta\cdot\nabla q =  k^2\nabla m\times(\eta + R_\zeta) + k^2 (1-n)(i\zeta\times\eta + Q).
\end{equation}
If $R_\zeta\in H^s_\delta$ and $|\zeta|$ is large, \eqref{G_bd} implies that $I + k^2G_\zeta (1-n)$ is invertible and the $\|Q\|_{H^s_\delta}$ is bounded by a constant independent of $\zeta$.

Therefore, by applying Lemma \ref{thm:lemma1} to \eqref{Maxwell7} and substituting the solution $R_\zeta$ in to the right hand side of \eqref{Maxwell7} recursively, $R_\zeta$ is given as a summation of a power series of $G_\zeta$, which converges  due to \eqref{G_bd}. This proves that  \eqref{Maxwell7} has a unique solution in $H^s_\delta(\RR^3)$. Furthermore, by Lemma \ref{thm:lemma1}, we see that
\begin{align}
	\|R + n^{-1/2}G_\zeta(n^{1/2} &[ -\alpha\times(\tilde\nabla\times R_\zeta)  \label{Maxwell8}	- (R_\zeta\cdot\nabla)\alpha \\&- \tilde{\nabla}(\alpha\cdot\eta) + k^2(1-n)(\eta + R_\zeta)])\|_{H^s_\delta} = O\left(\frac{1}{|\zeta|^2}\right).\notag
\end{align}
Substituting $\tilde{\nabla}(\alpha\cdot \eta) = (\nabla + i\zeta)(\alpha\cdot \eta)$, \eqref{Maxwell8} implies that
\begin{align}
  \|R + &in^{-1/2}G_\zeta(n^{1/2} \alpha\cdot\eta)\zeta\|_{H^s_\delta} \leq \|n^{-1/2}G_\zeta(n^{1/2}[ -\alpha\times(\tilde\nabla\times R_\zeta)	\\& - (R_\zeta\cdot\nabla)\alpha - \nabla\alpha\cdot\eta) + k^2(1-n)(\eta + R_\zeta)])\|_{H^s_\delta} + O\left(\frac{1}{|\zeta|^2}\right).\notag \\& = O\left(\frac{1}{|\zeta|}\right)
\end{align}
This complete the proof.

\end{proof}

By applying Sobolev embedding theorem to on a bounded domain, we have the estimate in $H^s(\Omega)$ and $C^d$, when $s = \frac{3}{2}+d+\iota$, $\iota>0$.
\begin{coro}
Let $\Omega$ be an open, bounded domain. With same hypotheses as the previous proposition, we then have
\begin{equation}
	\|R_\zeta - in^{-1/2}G_\zeta(n^{1/2}\alpha\cdot \eta) \zeta\|_{H^s(\Omega)} \leq \frac{C}{|\zeta|},
\end{equation}
for some positive constant $C$ independent of $\zeta$. Moreover, when $s = \frac{3}{2}+d+\iota$, we also have
\begin{equation}
\label{Maxwell19}
	\|R_\zeta - in^{-1/2}G_\zeta(n^{1/2}\alpha\cdot \eta) \zeta\|_{C^d(\Omega)} \leq \frac{C}{|\zeta|},
\end{equation}
for some positive constant $C$.
\end{coro}

\begin{prop} \label{thm:CGO2}
Suppose $\zeta\in\CC^3\backslash\RR^3$, $\eta\in\CC^3$, satisfy $\zeta\cdot\zeta = k^2$ and $\zeta\cdot\eta = 0$ such that as $|\zeta|\rightarrow \infty$ the limits $\zeta/|\zeta|$ and $\eta$ exist and,
\begin{equation}
\label{prop_limit}
	|\zeta/|\zeta| - \zeta_0| = O\left(\frac{1}{|\zeta|}\right),\quad |\eta - \eta_0| = O\left(\frac{1}{|\zeta|}\right).
\end{equation}
$R_\zeta\in H^s_\delta(\RR^3)$ is the unique solution to \eqref{Maxwell5} in Proposition \ref{thm:CGO}. Let $s = \frac{3}{2} + d + \iota$. For $|\zeta|$ large,
\begin{equation}
	\|R_\zeta - i|\zeta|n^{-1/2}G_\zeta(n^{1/2}\alpha\cdot \eta_0) \zeta_0\|_{C^d(\Omega)} = O\left(\frac{1}{|\zeta|}\right).
\end{equation}
\end{prop}
\begin{proof}
The proof follows directly by substituting \eqref{prop_limit} into \eqref{Maxwell19}.
\end{proof}

We choose the specific sets of $\zeta, \eta$ as in \cite{Colton1992}. Precisely, Let $h$ be a small real parameter and choose arbitrary $a\in\RR$. We define $\zeta_1,\zeta_2$ and $\eta_1,\eta_2$ by
\begin{equation}\label{eq:zetaeta}
\begin{array}{rcl}
	\zeta_1 &=& (a/2, i\sqrt{1/h^2 + a^2/4 -k^2}, 1/h),\\[6pt]
	\zeta_2 &=& (a/2, -i\sqrt{1/h^2 + a^2/4 -k^2}, -1/h),\\[6pt]
	\eta_1  &=& \frac{1}{\sqrt{1/h^2 + a^2}}(1/h,0,-a/2),\\[6pt]
	\eta_2  &=& \frac{1}{\sqrt{1/h^2 + a^2}}(1/h,0,a/2),
\end{array}
\end{equation}
and note that
\begin{equation}
\begin{array}{rcl}
	\displaystyle\lim_{c\rightarrow\infty} \eta_j = \eta_0 &:=& (1,0,0), \quad j = 1,2,\\[6pt]
	\displaystyle\lim_{c\rightarrow\infty} \zeta_1/|\zeta_1| &=&\zeta_0 := \frac{1}{\sqrt{2}}(0,i,1),\\[6pt]
	\displaystyle\lim_{c\rightarrow\infty} \zeta_2/|\zeta_2| &=& -\zeta_0,
\end{array}
\end{equation}
and
\begin{equation}
	\zeta_1 + \zeta_2 = (a, 0, 0),\quad \zeta_0\cdot\zeta_0 = 0,\quad \eta_0\cdot\zeta_0 = 0.
\end{equation}
Proposition \ref{thm:CGO2} implies that
\begin{equation}
(\eta_1 + R_{\zeta_1})\cdot(\eta_2 + R_{\zeta_2}) = 1 + o(1)
\end{equation}
in the $C^d$ norm over bounded domain $\Omega$.

\subsection{Construction of vector fields and uniqueness result}\label{se:unique}

Let us now consider the reconstruction of $(L,\sigma)$. Assume $E_j$, $j=1,2$, be two complex solutions to
\begin{equation}
\label{Maxwell_E_j}
	\nabla\times\nabla\times E_j - k^2n E_j = 0\;\mathrm{in}\; \Omega,
\end{equation}
with the tangential boundary conditions
\begin{equation}
\label{Maxwell_E_bd}
	\quad tE_j = G_j\;\mathrm{on}\;\partial\Omega, 
\end{equation}
with $G_j$ well-chosen boundary values and $j =1,2$. We will see that
\begin{equation}
	\label{eq:E1_E2}
	\nabla\times\nabla\times E_1\cdot E_2 - \nabla\times\nabla\times E_2\cdot E_1 = 0.
\end{equation}
Let $D_j = LE_j$, $j=1,2$, be the internal complex-valued measurements. Assume $L\in C^{d+1}(\Omega)$ is non-vanishing. Substituting $E_j = D_j/L$ in \eqref{eq:E1_E2}, we have, after some algebraical calculation,
\begin{equation}
	\label{eq:vector_field}
	\beta\cdot\nabla L + \gamma L = 0,
\end{equation}
where
\begin{eqnarray}
\beta &=& \chi(x) [(\nabla D_1)D_2 - (\nabla D_2)D_1] + [(\nabla\cdot D_1)D_2 - (\nabla\cdot D_2)D_1] \label{eq:beta}\\
	&& - 2[(\nabla D_1)^TD_2 - (\nabla D_2)D_1],  \notag\\
\gamma &=& \chi(x) [\nabla(\nabla\cdot D_1)\cdot D_2 - \nabla(\nabla\cdot D_2)\cdot D_1] + [\nabla^2D_1\cdot D_2 - \nabla^2D_2\cdot D_1]. \label{eq:gamma}
\end{eqnarray}
Here, $\chi(x)$ is a smooth known complex-valued function with $|\chi(x)|$ uniformly bounded from below by a positive constant on $\bar{\Omega}$.

To show the transport equation \eqref{eq:vector_field} has a unique solution, it suffices to prove that the direction of the vector field $\beta$ is close to a constant and thus the integral curves of $\beta$ connects every internal point to two boundary points.

Let $\check E_1, \check E_2$ be two CGOs with parameters $\zeta_1,\zeta_2$ and $\eta_1,\eta_2$ defined in \eqref{eq:zetaeta}, i.e.,
\begin{equation}
	\check E_j = e^{i\zeta_j\cdot x}(\eta_j + R_{\zeta_j}), \quad j= 1,2.
\end{equation}
Let $\check{D}_j = L \check{E}_j$, $j=1,2$, be the corresponding internal data. By choosing $\chi(x) = -e^{-i(\zeta_1+\zeta_2)\cdot x}\frac{h}{4\sqrt{2}}$ and substituting $\check{D}_j$ into \eqref{eq:beta}, we can analyze the asymptotic behavior of the vector field $\check{\beta}$ as $|\zeta_j|\rightarrow\infty$, or equivalently, $h\rightarrow 0$. Indeed, we have
\begin{equation}
	\nabla \check{D}_j = e^{i\zeta_j\cdot x}[(\eta_j + R_{\zeta_j})(\nabla L)^T + i \mu(\eta_j + R_{\zeta_j}) \zeta_j^T + \mu \nabla(\eta_j + R_{\zeta_j})], \quad j=1,2,
\end{equation}
and
\begin{eqnarray}
	&&\chi(x)(\nabla \check D_1)\check D_2 \notag\\
	&=&	-\frac{Lh}{4\sqrt{2}}[(\eta_1 + R_{\zeta_1})(\nabla L)^T + i L(\eta_1 + R_{\zeta_1}) \zeta_1^T + \mu \nabla(\eta_1 + R_{\zeta_1})](\eta_2+R_{\zeta_2})\phantom{space} \notag\\
	&=&	-\frac{Lh}{4\sqrt{2}}[(\eta_1+R_{\zeta_1})[\nabla L^T(\eta_2 + R_{\zeta_2}) + iL\zeta_1^T (\eta_2 + R_{\zeta_2})] \label{eq:beta1} \\
	&& + \mu\nabla(\eta_1 + R_{\zeta_1}) (\eta_2 + R_{\zeta_2})] \notag\\
	&=& -i\frac{L^2h}{4\sqrt{2}}(\eta_1+R_{\zeta_1})[\zeta_1\cdot(\eta_2+R_{\zeta_2})] + O(h). \notag
\end{eqnarray}
Therefore, by Proposition \ref{thm:CGO2}, $\chi(x)(\nabla \check D_1)\check D_2 \rightarrow 0$ in $C^d(\Omega)$ norm as $h\rightarrow 0$ on a bounded domain $\Omega$. Similarly,
\begin{eqnarray}
	&&\chi(x)(\nabla \check D_1)^T\check D_2 \notag\\
	&=&	-\frac{Lh}{4\sqrt{2}}[\nabla L (\eta_1 + R_{\zeta_1})^T(\eta_2 + R_{\zeta_2}) + i L \zeta_1 (\eta_1 + R_{\zeta_1})^T(\eta_2 + R_{\zeta_2}) \notag\\
	&& + L \nabla(\eta_1 + R_{\zeta_1})^T (\eta_2 + R_{\zeta_2})] \label{eq:beta2}\\
	&=& -iL^2\frac{\zeta_1h}{4\sqrt{2}}  (\eta_1+R_{\zeta_1})\cdot (\eta_2+R_{\zeta_2}) + O(h) \notag\\
	&\rightarrow& -\frac{iL\zeta_0}{4} \quad\mathrm{in}\; C^d(\Omega)\;\; \mathrm{as}\;\; h \rightarrow 0.\notag
\end{eqnarray}
More calculation gives that
\begin{eqnarray}
	&& \chi(x)(\nabla\cdot \check D_1)\check D_2 \notag\\
	&=& -\frac{L}{4\sqrt{2}c}[(\nabla\mu\cdot(\eta_1 + R_{\zeta_1}))(\eta_2 + R_{\zeta_2}) + iL(\zeta_1\cdot(\eta_1 + R_{\zeta_1}))(\eta_2 + R_{\zeta_2}) \label{eq:beta3}\\
	&& + L\nabla\cdot(\eta_1 + R_{\zeta_1})(\eta_2 + R_{\zeta_2})]\notag\\
	&\rightarrow& 0 \quad\mathrm{in}\; C^d(\Omega)\;\; \mathrm{as}\;\; h \rightarrow 0.\notag
\end{eqnarray}
By substituting \eqref{eq:beta1}, \eqref{eq:beta2} and \eqref{eq:beta3} into \eqref{eq:beta}, we have 
\begin{equation}
\label{eq:beta4}
	\lim_{h\rightarrow 0}\|\check{\beta} - L^2\zeta_0 \|_{C^d(\Omega)} = 0,
\end{equation}
i.e., the vector fields have approximately constant directions for small $h$ and their integral curves connect every internal point to two boundary points. Thus, the transport equation \eqref{eq:vector_field} admits a unique solution. 

To see the dependence of vector fields on the boundary conditions, we need to introduce a regularity theorem of Maxwell's equations. Let $tE$ be the tangential boundary condition of $E$. Define the Div-spaces as
\begin{equation}
	H^s_{\mathrm{Div}}(\Omega) = \{u\in H^s\Omega^1(\Omega): \mathrm{Div}(tu)\in H^{s-1/2}(\partial\Omega)\},
\end{equation}
\begin{equation}
	TH^s_{\mathrm{Div}}(\partial\Omega) = \{g\in H^s\Omega^1(\partial\Omega):\mathrm{Div}(g)\in H^s(\partial\Omega)\},
\end{equation}
where $H^s\Omega^1(\Omega)$ is a space of vector functions of which each component is in $H^s(\Omega)$. These are Hilbert spaces with norms
\begin{equation}
\|u\|_{H^s_{\mathrm{Div}}(\Omega)} = \|u\|_{H^s(\Omega)} + \|\mathrm{Div}(tu)\|_{H^{s-1/2}(\partial\Omega)},
\end{equation}
\begin{equation}
\|g\|_{TH^s_{\mathrm{Div}}(\partial\Omega)} = \|g\|_{H^s(\partial\Omega)} + \|\mathrm{Div}(g)\|_{ H^s(\partial\Omega)}
\end{equation}
It is clear that $t(H^s_{\mathrm{Div}}(\Omega)) = TH^{s-1/2}_{\mathrm{Div}}(\partial\Omega)$.

\begin{thm}\label{thm:regularity}
Let $\epsilon,\mu \in C^s$, $s>2$, be positive functions. There is a discrete subset $\Sigma\subset\CC$ such that if $\omega$ is outside this set, then one has a unique solution $E\in H^d_{\mathrm{Div}}$ to \eqref{Maxwell_E1} given any tangential boundary condition $G\in TH^{s-1/2}_{\mathrm{Div}}(\partial\Omega)$. The solution satisfies
\begin{equation}
\|E\|_{H^s_{\mathrm{Div}}(\Omega)} \leq C \|G\|_{TH^{s-1/2}_{\mathrm{Div}}(\partial\Omega)}
\end{equation}
with $C$ independent of $G$. 
\end{thm}

Note that when the tangential boundary condition is prescribed by CGOs, i.e., $\check G_j = t\check{E}_j$, $j=1,2$. By Theorem \ref{thm:regularity}, $E_j$ is the unique solution to \eqref{Maxwell_E_j} and \eqref{Maxwell_E_bd}. Then the corresponding vector field $\check{\beta}$ defined in \eqref{eq:beta} satisfies \eqref{eq:beta4}, which implies that the direction of $\check\beta$ is close to constant direction and thus its integral curves connect every internal point to two boundary points. Therefore, \eqref{eq:vector_field} admits a unique solution. 

Furthermore, Theorem \ref{thm:regularity} also allows one to relax the boundary condition $\check{G}_j = t\check E_j$ and still to get the uniqueness of the solution to \eqref{eq:vector_field}.

\begin{prop} \label{thm:betafield}
Under the assumption of Theorem \ref{thm:regularity}, when $G_j$ is in a neighborhood of $\check G = t\check{E}$ in $C^{d+3}(\partial\Omega)$, $j=1,2$, the corresponding vector field  $\beta$ defined in \eqref{eq:beta} satisfies
\begin{equation}
\label{eq:beta5}
	\|\beta - L^2\zeta_0 \|_{C^d(\Omega)} = O(h),
\end{equation}
for small $h$.
\end{prop}
\begin{proof}
Let $s = \frac{5}{2} + d + \iota$. Notice that $C^{d+3}(\partial\Omega) \subset H^{d+3}(\partial\Omega) \subset H^{\frac{5}{2}+d+\iota}(\partial\Omega)$. By the Sobolev embedding theorem and Theorem \ref{thm:regularity}, we have that
\begin{equation}
\label{eq:regularity1}
  \|E\|_{C^{d+1}(\Omega)} \leq C\|G\|_{C^{d+3}(\partial\Omega)}.
\end{equation}
Let us now define boundary conditions $G_j\in C^{d+3}(\partial\Omega)$, $j=1,2$, such that
\begin{equation}
\label{eq:regularity2}
  \|G_j - t\check E_j\|_{C^{d+3}(\partial\Omega)} \leq \varepsilon,
\end{equation}
for some $\varepsilon > 0$ sufficiently small. let $E_j$ be the solution to \eqref{Maxwell_E1} and \eqref{Maxwell_E2} with $tE_j = G_j$. By \eqref{eq:regularity1}, we thus have
\begin{equation}
\label{eq:regularity3}
  \|E _j- \check E_j\|_{C^{d+1}(\Omega)} \leq C\varepsilon,
\end{equation}
for some positive constant $C$. Define the complex valued internal data $D_j=LE_j$. We deduce that
\begin{equation}
\label{eq:regularity4}
  \|D_j - \check D_j\|_{C^{d+1}(\Omega)} \leq C_0\varepsilon,
\end{equation}
for $C_0>0$.
Define $\beta$ by \eqref{eq:beta}. We can easily deduce \eqref{eq:beta5} from \eqref{eq:beta4} and \eqref{eq:regularity4}. This finishes the proof.
\end{proof}

Recall $\mathcal{M}$ is the parameter space of $(L,\sigma)$ defined in \eqref{eq:para_space} and $h$ is the parameter in \eqref{eq:zetaeta}. We are in the place to prove Theorem \ref{thm:unique}.

\begin{proof}[Proof of Theorem \ref{thm:unique}]
By Proposition \ref{thm:betafield}, we choose the set of illuminations as a neighborhood of $(\check{G}_j) = (t\check{E}_j)$ in $(C(\partial\Omega))^2$. Since the measurements $D=\tilde{D}$, we have that $L$ and $\tilde{L}$ solve the same transport equation \eqref{eq:vector_field} while $L=\tilde{L} = D/G$ on $\partial\Omega$. As $\beta$ satisfies \eqref{eq:beta5}, we deduce that $L=\tilde{L}$ since the integral curves of $\beta$ map any $x\in\Omega$ to the boundary $\partial\Omega$. More precisely, consider the flow $\theta_x(t)$ associated to $\beta$, i.e., the solution to 
\begin{equation}
\label{eq:beta6}
	\dot{\theta}_x(t) = \beta(\theta_x(t)), \quad \theta(0) = x\in\bar{\Omega}.
\end{equation}
By the Picard-Lindel\"of theorem, \eqref{eq:beta6} admits a unique solution since $\beta$ is of class $C^1(\Omega)$. For $x\in\Omega$, let $x_\pm(x)\in\partial\Omega$ and $t_\pm(x)>0$ such that
\begin{equation}
\label{eq:beta7}
	\theta_x(t_\pm(x)) = x_\pm(x) \in\partial\Omega.
\end{equation}
By the method of characteristics, the solution $L$ to the transport equation \eqref{eq:vector_field} is given by
\begin{equation}
\label{eq:beta8}
	L(x) = L_0(x_\pm(x))e^{-\int_0^{t_\pm(x)}\gamma(\theta_x(s))ds}.
\end{equation}
The solution $\tilde{L}$ is given by the same formula since $\theta_x(t) = \tilde{\theta}_x(t)$. This implies $E_j = \check{E}_j = D_j/L$, $j=1,2$. By the choice of illuminations, we have $|E_j|\neq 0$ due to \eqref{eq:regularity3} and $|\check{E}_j|\neq 0$. Therefore, $k^2n$ is uniquely determined by \eqref{Maxwell_E_j} and thus $\sigma = \tilde{\sigma}$.
\end{proof}

\subsection{Vector fields and stability result}\label{se:stab}

Recall that $\theta_x(t)$ is the flow associated with $\beta$. From the equality
\begin{equation}
\label{eq:beta_9}
\theta_x(t) - \tilde{\theta}_x(t) = \int_0^t [\beta(\theta_x(s)) - \tilde{\beta}(\tilde{\theta}_x(s))]ds,
\end{equation}
and using the Lipschitz continuity of $\beta$ and Gronwall's lemma, we deduce the existence of a constant $C$ such that
\begin{equation}
\label{eq:beta10}
	|\theta_x(t) - \tilde{\theta}_x(t)| \leq Ct\|\beta - \tilde{\beta}\|_{C^0(\bar{\Omega})},
\end{equation}
when $\theta_x(t)$ and $\tilde{\theta}_x(t)$ are in $\bar{\Omega}$. The inequality \eqref{eq:beta10} is uniform in $t$ as all characteristics exit $\bar{\Omega}$ in finite time.

To see higher order estimates, we define $W:=D_x\theta_x(t)$, which solves the equation, $\dot{W}=D_x\beta(\theta_x)W$, with $W(0) = I$. Define $\tilde{W}$ similarly. By using Gronwall's lemma again, we deduce that 
\begin{equation}
\label{eq:beta11}
	|W - \tilde{W}| \leq Ct\|D_x\beta - D_x\tilde{\beta}\|_{C^0(\bar{\Omega})}
\end{equation}
when $\theta_x(t)$ and $\tilde{\theta}_x(t)$ are in $\bar{\Omega}$. Since $\beta$ and $\tilde{\beta}$ are of class $C^d(\bar{\Omega})$, then we obtain iteratively that 
\begin{equation}
\label{eq:beta12}
	|D_x^{d-1}\theta_x(t) - D_x^{d-1}\tilde{\theta}_x(t)| \leq Ct\|D_x\beta - D_x\tilde{\beta}\|_{C^{d-1}(\bar{\Omega})},
\end{equation}
when $\theta_x(t)$ and $\tilde{\theta}_x(t)$ are in $\bar{\Omega}$.


Recall that $\Omega_1$ is defined to be the subset of $\Omega$ by removing a neighborhood of each tangent point of $\partial\Omega$ with respect to $\zeta_0$.

\begin{lemma}\label{thm:stab1}
Let $\Omega$ be an open bounded and convex subset in $\RR^3$ with $C^d$ boundary. Let $d\geq 2$ and assume $\beta$ and $\tilde{\beta}$ are $C^d(\bar{\Omega})$ vector fields which satisfy \eqref{eq:beta5}. Restricting to $\Omega_1$, we have that
\begin{equation}
\label{eq:beta13}
	\|x_+ - \tilde{x}_+\|_{C^{d-1}(\Omega_1)} + \|t_+ - \tilde{t}_+\|_{C^{d-1}(\Omega_1)} \leq C \|\beta_+ - \tilde{\beta}_+\|_{C^{d-1}(\Omega_1)},
\end{equation}
where $C$ is a constant depending on $\Omega$.
\end{lemma}

This lemma is similar to the lemma 3.8 in \cite{Bal2010} and the lemma 4.1 in \cite{ChenYang2012}, but uses a different proof.

\begin{proof}
For $x\in\Omega_1$,  let $\theta_x(t)$ and $\tilde{\theta}_x(t)$ be two flows associated to vector fields $\beta$ and $\tilde{\beta}$, respectively. Denote $A:=\theta_x(t_+(x))\in\partial\Omega$ and $B:=\tilde{\theta}_x(\tilde{t}_+(x))\in\partial\Omega$. Without loss of generality, we assume $t_+(x)\leq \tilde{t}_+(x)$. We also denote $C:=\tilde{\theta}_x(t_+(x))\in\Omega$. To simplify the writing, let $\delta := \|\beta - \tilde{\beta}\|_{C^d(\Omega_1)}$.

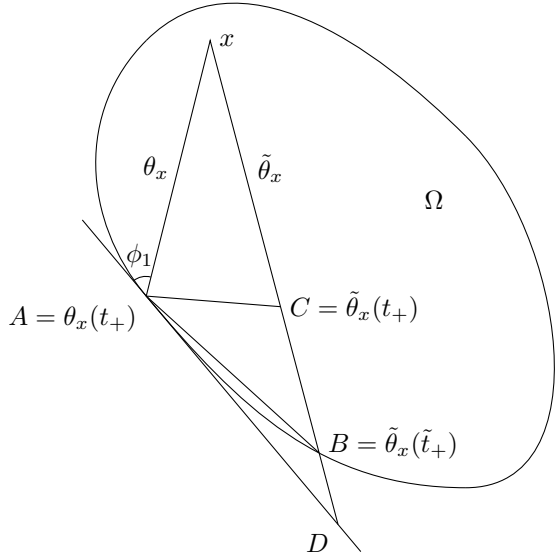
\begin{figure}[ht]
\centering
\begin{tikzpicture}[scale=0.85]
\draw (1, 4)       to [out=-50,in=180]   (6, 1);
\draw (6, 1)       to [out=0, in=-45]    (6, 6.5);
\draw (6, 6.5)     to [out=135, in=45]   (1, 8);
\draw (1, 8)       to [out=-135, in=130] (1, 4); 
\draw (0, 5.192)   -- (4.356, 0);
\draw (3.7, 1.55)  -- (1, 4) --(2, 8) -- (4, 0.425);
\draw (1, 4)       -- (3.1, 3.834);
\draw (0.8, 4.238) arc [radius=0.3, start angle=130, end angle= 75];
\node [right]      at (2,8) {$x$};
\node [left]       at (1.5,6) {$\theta_x$};
\node [right]      at (2.6,6) {$\tilde{\theta}_x$};
\node [below left] at (1,4) {$A=\theta_x(t_+)$};
\node [right]      at (3.7,1.7) {$B=\tilde{\theta}_x(\tilde{t}_+)$};
\node [right]      at (3.1, 3.834) {$C=\tilde{\theta}_x(t_+)$};
\node [below left] at (4,0.425) {$D$};
\node [above]      at (0.9,4.3) {$\phi_1$};
\node              at (5.5, 5.5) {$\Omega$};
\end{tikzpicture}
\caption{Vector fields $\beta$ and $\tilde{\beta}$}
\label{fig:vector_field}
\end{figure}

We first want to show that the angle $\angle AxB$ is controlled by
\begin{equation}
\label{eq:beta14}
\angle AxB \leq C_1\delta + C_2h,
\end{equation}
for some $C_1, C_2$. Indeed, by applying \eqref{eq:beta10} and sine theorem, we can see that $\angle AxC$ is bounded by $C_1\delta$. Also notice that $\tilde\beta$ satisfies \eqref{eq:beta5}. Therefore, similar argument shows that, for any $t_1, t_2$, the angle between the vector from $x$ to $\tilde{\theta}_x(t_1)$ and the vector from $x$ to $\tilde{\theta}_x(t_2)$ is bounded by $C_2h$. Thus $\angle CxB\leq C_2 h$. This proves \eqref{eq:beta14}.

By the definition of $\Omega_1$, a neighborhood of the boundary point at which the tangent plane of $\partial\Omega$ is parallel to $\zeta_0$ is removed. therefore, there exists a constant value $\phi_0>0$ depending only on $\Omega_1$ such that, for any $x\in\Omega_1$, $\phi_1\geq\phi_0$, where $\phi_1$ is the angle between the vector $\overline{xA}$ and the tangle plane of $\partial\Omega$ at $A$, as in Fig. \ref{fig:vector_field}. Then by \eqref{eq:beta14}, when $\delta$ and $h$ are so small that $\phi_0':=\phi_0 - C_1\delta - C_2h > 0$, the extension of $\overline{xB}$ will intersect the tangent plane of $\partial\Omega$ at $A$, with intersection point $D$. Then it is easy to check that
\begin{equation}
\label{eq:beta15}
\angle ABC >\angle ADx = \phi_1 - \angle AxB > \phi_0' > 0.
\end{equation} The sine theorem gives that
\begin{equation}
\label{eq:beta16}
|AB| = \frac{|AC|}{\sin(\angle ABC)}\sin(\angle ACB).
\end{equation}
\eqref{eq:beta10} directly implies $|AB| = |x_+ - \tilde{x}_+|\leq C'\delta$. 

Since $\beta,\tilde{\beta}\in C^d(\Omega)$ and $\partial\Omega$ is of class $C^d$, it is clear that $\angle ABC$ and $\angle ACB$ are $C^d$ functions with respect to $x\in\Omega$. By differentiating \eqref{eq:beta16} and applying \eqref{eq:beta12}, we get higher order estimates
\begin{equation}
\label{eq:beta17}
\|x_+ - \tilde{x}_+\|_{C^{d-1}(\Omega_1)} \leq C''\delta.
\end{equation}

To see the second part in \eqref{eq:beta13}, we have that
\begin{equation}
\label{eq:beta18}
|CB| = \int_{t_+(x)}^{\tilde{t}_+(x)} \tilde{\beta}(\tilde{\theta}_x(d))ds = |\tilde{\beta}(\tilde{\theta}_x(\tau))|(\tilde{t}_+(x) - t_+(x)),
\end{equation}
for $t_+(x) \leq \tau \leq \tilde{t}_+(x)$. Similar argument shows the  estimate of $t_+-\tilde{t}_+$ in \eqref{eq:beta13}.
\end{proof}

\begin{prop}\label{thm:stab2}
Let $d\geq 1$. Let $L$ and $\tilde{L}$ be solutions to \eqref{eq:vector_field} corresponding to coefficients $(\beta,\gamma)$ and $(\tilde{\beta}, \tilde{\gamma})$, respectively, where \eqref{eq:beta5} holds for both $\beta$ and $\tilde{\beta}$.

Let $L_0 = L|_{\partial\Omega}$ and $\tilde{L}_0 = \tilde{L}_{\partial\Omega}$, thus $L_0,\tilde{L}_0 \in C^d(\partial\Omega)$. We also assume $h$ is sufficiently small and $\Omega$ is convex. Then there is a constant $C$ such that restricting to $\Omega_1$
\begin{align}
\label{eq:stable1}
\|L-\tilde{L}\|_{C^{d-1}(\Omega_1)} \leq & C\|L_0\|_{C^d(\partial\Omega_1)}[\|\beta-\tilde{\beta}\|_{C^{d-1}(\Omega_1)} \\ 
& + \|\gamma-\tilde{\gamma}\|_{C^{d-1}(\Omega_1)}] + C\|L_0-\tilde{L}_0\|_{C^d(\partial\Omega_1)} \notag
\end{align}
\end{prop}

The proposition and the proof follows proposition 4.2 in \cite{ChenYang2012}.
\begin{proof}

By the method of characteristics, $L(x)$ is determined explicitly in \eqref{eq:beta8}, while $\tilde{L}$ has a similar expression. We thus have
\begin{align}\label{eq:stable2}
|L(x) - \tilde{L}(x)| \leq & |(L_0(x_+(x)) - \tilde{L}_0(\tilde{x}_+(x)))e^{-\int_0^{t_+(x)}\gamma(\theta_x(s))ds}| \\ 
& + |\tilde{L}_0(\tilde{x}_+(x))(e^{-\int_0^{t_+(x)}\gamma(\theta_x(s))ds} - e^{-\int_0^{\tilde{t}_+(x)}\tilde{\gamma}(\tilde{\theta}_x(s))ds})|.\notag
\end{align}
Applying Lemma \ref{thm:stab1}, we deduce that
\begin{equation}\label{eq:stable3}
|D_x^{d-1}[L_0(x_+(x)) - \tilde{L}_0(\tilde{x}_+(x))]| \leq \|L_0-\tilde{L}_0\|_{C^{d-1}(\partial\Omega)} + C \|L_0\|_{C^{d-1}(\partial\Omega)}\|\beta-\tilde{\beta}\|_{C^{d-1}(\Omega_1)}.
\end{equation}
This proves $L_0(x_+(x))$ is stable. To consider the second term, by the Leibniz rule, it is sufficient to prove the stability result for $\int_0^{t_+(x)}\gamma(\theta_x(s))ds$.

Assume without loss of generality that $t_+(x) < \tilde{t}_+(x)$. Then by applying \eqref{eq:beta12}, we have
\begin{align}
\int_0^{t_+(x)} [\gamma(\theta_x(s)) - \tilde{\gamma}(\tilde{\theta}_x(s))]ds &= \int_0^{t_+(x)} [\gamma(\theta_x(s)) - \gamma(\tilde{\theta}_x(s)) + (\gamma - \tilde{\gamma})\tilde{\theta}_x(s)]ds \notag\\
& \leq C\|\gamma\|_{C^0(\Omega_1)}\|\beta - \tilde{\beta}\|_{C^0(\Omega_1)} + C\|\gamma - \tilde{\gamma}\|_{C^0(\Omega_1)}. \notag
\end{align}
Derivatives of order $d-1$ of the above expression are uniformly bounded since $t_+(x)\in C^{d-1}(\Omega_1)$, $\gamma$ has $C^d$ derivatives bounded on $\Omega$ and $\theta_x(t)$ is stable as in \eqref{eq:beta12}.

It remains to handle the term
\begin{equation*}
v(x) := \int_{t_+(x)}^{\tilde{t}_+(x)} \tilde{\gamma}(\tilde{\theta}_x(s))ds.
\end{equation*}
$\tilde{\beta}$ and $\tilde{\gamma}$ are of class $C^d(\Omega_1)$, so is the function $x\rightarrow\tilde{\gamma}(\tilde{\theta}_x(s))$. Derivatives of order $d-1$ of $v(x)$ involve terms of size $t_+(x) - \tilde{t}_+(x)$ and terms of form
\begin{equation*}
D_x^m\bigg(\tilde{t}_+ D_x^{d-1-m}\tilde{\gamma}(\tilde{\theta}_x(\tilde{t}_+)) - t_+ D_x^{d-1-m}\tilde{\gamma}(\tilde{\theta}_x(t_+))\bigg), \quad 0\leq m\leq d-1.
\end{equation*}
Since the function has $d-1$ derivatives that are Lipschitz continuous, we thus have
\begin{equation*}
|D_x^{d-1}v(x)| \leq C \|t_+ - \tilde{t}_+\|_{C^{d-1}(\Omega_1)}.
\end{equation*}
The rest of the proof follows Lemma \ref{thm:stab1}.
\end{proof}

Now we can prove the main stability theorem.

\begin{proof}[Proof of Theorem \ref{thm:stab}]
The first part follows directly from \eqref{eq:beta} and Proposition \ref{thm:stab2}. This also provides a stability result for $E_j = D_j/L$ as $L$ is non-vanishing. By choosing the boundary illuminations close to the boundary conditions of CGO solutions, \eqref{eq:regularity2} and \eqref{eq:regularity3} imply that $E_j$ is non-vanishing since the CGO solutions are non-vanishing. Thus \eqref{Maxwell_E1} gives the stability control of $k^2n$ and thus $\sigma$.
\end{proof}

\subsection{Stability with 6 complex internal data}\label{se:stab2}
Rather than applying the characteristics method to \eqref{eq:vector_field}, we can rewrite \eqref{eq:vector_field} into matrix form by introducing more internal measurements. We first construct proper CGO solutions. Let $j=1,2,3$ in this section. We can choose unit vectors $\zeta_0^j$ and $\eta_0^j$, such that $\zeta_0^j\cdot \zeta_0^j = 0$, $\zeta_0^j\cdot\eta_0^j=0$ and $\{\zeta_0^j\}$ are linearly independent. Also, choose $(\zeta_1^j, \zeta_2^j)$ and $(\eta_1^j, \eta_2^j)$ such that $|\zeta|:=|\zeta_1^j|=|\zeta_2^j|$,
\begin{align}\label{eq:stable4}
\lim_{|\zeta|\rightarrow\infty} \frac{\zeta_1^j}{|\zeta|} = \lim_{|\zeta|\rightarrow\infty} \frac{\zeta_2^j}{|\zeta|} = \zeta_0^j, \quad\mathrm{and}\quad \lim_{|\zeta|\rightarrow\infty} \eta^j = \eta_0^j.
\end{align}

We construct CGO solutions $\check{E}_1^j, \check{E}_2^j$ corresponding to $(\zeta_1^j,\eta_1^j)$ and $(\zeta_2^j, \eta_2^j)$. Let the boundary illuminations $G_1^j, G_2^j$ be chosen according to \eqref{eq:regularity2} for $\varepsilon$ small enough. The measured internal data are then given by $D_1^j, D_2^j$. Proposition \ref{thm:regularity} shows that the vector field defined by \eqref{eq:beta} satisfies that

\begin{equation}\label{eq:stable5}
	\|\beta^j - L^2\zeta_0^j\|_{C^d(\Omega)} \leq \frac{C}{|\zeta|}.
\end{equation}
While $|\zeta|$ is sufficiently large and $L\neq 0 $ on $\bar{\Omega}$, we obtain that the vector $\{\beta^j(x)\}$ are linear independent at every $x\in\Omega$. Thus matrix $(\beta^j(x))$ is invertible with inverse of class $C^d(\Omega)$. By constructing vector-valued function  $\Gamma(x)\in (C^d(\bar\Omega))^3$, the transport equation \eqref{eq:vector_field} now becomes the matrix equation
\begin{equation}\label{eq:stable6}
\nabla L + \Gamma(x) L = 0.
\end{equation}
Notice that $\Gamma(x)$ is stable under small perturbations in the data $D := (D^j_1, D_2^j)\in (C^{d}(\bar\Omega))^6$, i.e.,
\begin{equation}\label{eq:stable7}
\|\Gamma - \tilde{\Gamma}\|_{(C^{d-1}(\Omega))^3} \leq C \|D - \tilde{D}\|_{(C^{d}(\Omega))^6}.
\end{equation}

Assume $\Omega$ is connected and $L_0 = L|_{\partial\Omega}$ is known. Choose a smooth curve from $x\in\Omega$ to a point on the boundary. Restricting to the curve, \eqref{eq:stable6} is a stable ordinary differential equation. Keep the curve fixed. Let $L$ and $\tilde{L}$ be solutions to \eqref{eq:stable6} with respect to $\Gamma$ and $\tilde{\Gamma}$, respectively. By solving the equation explicitly and \eqref{eq:stable7}, we find that 
\begin{equation}
\label{eq:stable8}
\|L - \tilde{L}\|_{C^{d-1}(\Omega)} \leq C \|D - \tilde{D}\|_{(C^{d}(\Omega))^6}.
\end{equation}

\begin{proof}[Proof of Theorem \ref{thm:stab2n}]
The first result in \eqref{eq:stable9} is directly from \eqref{eq:stable8}. The proof of the stability of $\sigma$ is exactly the same as in the proof of theorem \ref{thm:stab}.
\end{proof}

\section{Acknowledgment}\label{se:acknowledge}
The authors thank Professor Gunther Uhlmann for suggesting this problem and for helpful discussions. The authors also thank Professor Maarten de Hoop for helpful discussions on the inverse problem models. The work of both authors was partly supported by NSF.


\begin{thebibliography}{99}
	\bibitem{Bal2010} G. Bal and G. Uhlmann, \emph{Inverse diffusion theory of photo-acoustics}, Inverse Problems, 26 (2010), 085010.
	
	\bibitem{Biot1956} M. A. Biot, “Theory of propagation of elastic waves in a fluid-saturated porous solid. I-Low-frequency range,” Journal of the Acoustical Society of America, vol. 28, pp. 168-178, 1956.
	
	\bibitem{Biot1956_2} M. A. Biot, “Theory of propagation of elastic waves in a fluid-saturated porous solid. II-High-frequency range,” Journal of the Acoustical Society of America, vol. 28, pp. 179-191, 1956.
	
	\bibitem{ChenYang2012} J. Chen and Y. Yang, \emph{Quantitative photo-acoustic tomography with partial data}, Inverse Problem, 28 (2012), 115014.
	
	\bibitem{Colton1992} D. Colton and L. \Paivarinta, \emph{The uniqueness of a solution to an inverse scattering problem for electromagnetic waves}, Arch. Rational Mech. Anal. 119 (1992), 59-70.
	
	\bibitem{Pride1994} S. R. Pride, \emph{Governing equations for the coupled electro-magnetics and acoustics of porous media}, Phys. Rev. B, 50 (1994), 15678-15696.
	
	\bibitem{Pride1996} S. R. Pride and M. W. Haartsen, \emph{Electroseismic wave properties}, Journal of the Acoustical Society of America, vol.	100, no. 3 (1996), pp. 1301–1315.
	
	\bibitem{Thom1993} A. Thompson and G. Gist, \emph{Geophysical applications of electro-kinetic conversion}, The Leading Edge, vol. 12 (1993), pp. 1169–1173
	
	\bibitem{Uhlmann1987} J. Sylvester and G. Uhlmann, \emph{A global uniqueness theorem for an inverse boundary value problem}, Ann. of Math., 125(1) (1987), pp. 153-169.1
	
	\bibitem{Williams2001} K. L. Williams, \emph{An effective density fluid model for acoustic propagation	in sediments derived from Biot theory}, J. Acoust. Soc. Am. Volume 110, Issue 5, pp. 2276-2281 (2001)
	
	\bibitem{Zhu1999} Z. Zhu, M. W. Haartsen, and M. N. Toks\"oz, “Experimental	studies of electro-kinetic conversions in fluid-saturated bore-hole models,” Geophysics, vol. 64, no. 5, pp. 1349-1356, 1999.
	
	\bibitem{Zhu2003} Z. Zhu and M. N. Toks\"oz, “Cross hole seismo-electric measurements in bore-hole models with fractures,” Geophysics, vol. 68, no. 5, pp. 1519-1524, 2003.
	
	\bibitem{Zhu2005} Z. Zhu and M. N. Toks\"oz, “Seismo-electric and seismo-magnetic measurements in fractured bore-hole models,” Geophysics, vol. 70, no. 4, pp. F45-F51, 2005.
	

\end{thebibliography}
\end{document}